\documentclass[reqno]{amsart}
\usepackage{amsfonts}%
\usepackage{mathrsfs}

\usepackage{cite}
\allowdisplaybreaks
\date{\today}

\theoremstyle{plain}
\newtheorem{thm}{Theorem}[section]

\newtheorem{prop}[thm]{Proposition}
\theoremstyle{definition}

\theoremstyle{remark}
\newtheorem{rem}{Remark}[section]
\numberwithin{equation}{section}

\renewcommand{\u}{{\bf u}}

\renewcommand{\H}{{\bf H}}
\newcommand{\g}{{\bf g}}

\newcommand{\dv}{{\rm div }}
\newcommand{\cu}{{\rm curl\, }}

\begin{document}

\title[Incompressible limit of MHD equations]
{Incompressible   limit of the compressible
magnetohydrodynamic equations with vanishing viscosity coefficients}

\author{Song Jiang}
\address{LCP, Institute of Applied Physics and Computational Mathematics, P.O.
 Box 8009, Beijing 100088, P.R. China}
 \email{jiang@iapcm.ac.cn}

\author{Qiangchang Ju}
\address{Institute of Applied Physics and Computational Mathematics, P.O.
 Box 8009-28, Beijing 100088, P.R. China}
 \email{qiangchang\_ju@yahoo.com}

 \author[Fucai Li]{Fucai Li$^\ast$}
\address{Department of Mathematics, Nanjing University, Nanjing
 210093, P.R. China}
 \email{fli@nju.edu.cn}
\thanks{$^*$ Corresponding  author}

\keywords{compressible MHD equations, ideal incompressible MHD equations,
 low Mach number, zero viscosities}

\subjclass[2000]{76W05, 35B40}

\begin{abstract}
This paper is concerned with the  incompressible  limit of the compressible
magnetohydrodynamic equations with vanishing viscosity coefficients and  general initial data  in the
whole space $\mathbb{R}^d$ $ (d=2$ or $3$). It is rigorously showed that, as the Mach
number, the shear viscosity coefficient and the magnetic diffusion
coefficient simultaneously go to zero,
 the weak solutions of the compressible magnetohydrodynamic equations
 converge to the strong solution of the ideal incompressible magnetohydrodynamic
 equations as long as the latter exists.

 \end{abstract}

\maketitle

\section{Introduction}
Magnetohydrodynamics (MHD) studies  the  dynamics of compressible
quasineutrally ionized fluids under the influence of electromagnetic
fields. The applications of MHD cover a very wide
range of physical objects, from liquid metals to cosmic plasmas.
The compressible viscous magnetohydrodynamic
equations in the isentropic case take the form (\!\cite{KL,LL,PD})
\begin{align}
&\partial_t \tilde{\rho}  +\dv(\tilde{\rho}\tilde{\u})=0, \label{a1a} \\
&\partial_t
(\tilde{\rho}\tilde{\u})+\dv(\tilde{\rho}\tilde{\u}\otimes\tilde{\u})+\nabla
\tilde{P}  =(\cu\tilde{ \H})\times \tilde{\H}+\tilde{\mu}\Delta\tilde{\u}
+(\tilde{\mu}+\tilde{\lambda})\nabla(\dv\tilde{\u}), \label{a1b} \\
&\partial_t \tilde{\H}
-\cu(\tilde{\u}\times\tilde{\H})=-\cu(\tilde{\nu}\,
\cu\tilde{\H}),\quad \dv\tilde{\H}=0.\label{a1c}
\end{align}
Here $x\in \mathbb{R}^d, d=2 $ or $3$,  $ t>0$,  the unknowns
$\tilde{\rho}$ denotes the density,
$\tilde{\u}=(\tilde{u}_1,\dots,\tilde{u}_d)\in
 \mathbb{R}^d$ the velocity, and $\tilde{\H}=(\tilde{H}_1,\dots, \tilde{H}_d)\in  \mathbb{R}^d$ the
magnetic field, respectively. The constants  $\tilde{\mu}$ and
$\tilde{\lambda}$ are the shear and bulk viscosity coefficients of
the flow, respectively, satisfying $\tilde{\mu}>0$ and
 $2\tilde{\mu}+d\tilde{\lambda}\geq 0$; the constant $\tilde{\nu}>0$ is the magnetic
diffusivity acting as a magnetic diffusion coefficient of the
magnetic field. $\tilde{P}(\tilde{\rho})$ is the pressure-density
function  and here we consider the case
\begin{equation}
\tilde{P}(\tilde{\rho})=a \tilde{\rho}^\gamma, \quad a>0, \gamma >1.
\label{aad}
\end{equation}

There are a lot of studies on the compressible MHD equations in the
literature. Here we mention some results on the multi-dimensional
case. For the two-dimensional case, Kawashima~\cite{K} obtained the
global existence of smooth solutions to the general
electromagnetofluid equations when the initial data are small
perturbations of some given constant state. For the
three-dimensional compressible MHD equations,
  Umeda, Kawashima and Shizuta~\cite{UKS} obtained the global
existence and the time decay of smooth solutions to the linearized
MHD equations. Li and Yu \cite{LY} obtained the optimal decay rate
of classical solutions to the compressible MHD equations around a
constant equilibrium. The local strong solution to the
  compressible MHD equations  with general initial data was
obtained by Vol'pert and Khudiaev~\cite{VK}, and Fan and Yu \cite{FY1}.
Recently, Hu and Wang~\cite{HW1,HW2}, and Fan and Yu \cite{FY2}
established the existence of global weak solutions to the compressible
MHD equations with general initial data; while in \cite{ZJX} Zhang, Jiang and Xie
discussed a MHD model describing the screw pinch problem in plasma physics
and showed the global existence of weak solutions with symmetry.

From the physical point of view, one can formally derive the
incompressible models from the compressible ones when the Mach
number goes to zero and the density becomes almost constant.
Recently, Hu and Wang \cite{HW3} obtained the convergence of weak
solutions of the compressible MHD equations \eqref{a1a}-\eqref{a1c}
to the weak solution of the incompressible viscous MHD equations. In
\cite{JJL}, the authors  have employed the modulated energy method
to verify the limit of weak solutions of the compressible MHD
equations \eqref{a1a}-\eqref{a1c} in the torus to the strong
solution of the incompressible viscous or partial viscous MHD
equations (the shear viscosity coefficient is zero but magnetic
diffusion coefficient is a positive constant). It is worth
mentioning that the analysis in \cite{JJL} turns out that the
magnetic diffusion is very important to control the interaction
between the oscillations and the magnetic field in the torus. The
rigorous justification of convergence for compressible MHD equations
  to  the \emph{ideal incompressible MHD equations} (inviscid, incompressible MHD
equations)  in the torus is left open in \cite{JJL} for general
initial data.

The purpose of this paper is to derive the ideal incompressible
MHD equations  from the compressible MHD equations
\eqref{a1a}-\eqref{a1c} in the whole space $\mathbb{R}^d$ $ (d=2$ or $3$) with general initial data. In fact, when the
viscosities (including the shear viscosity coefficient and the
magnetic diffusion coefficient) also go to zero, we lose the spatial
compactness property of the velocity and the magnetic field, and the
arguments in \cite{HW3} can not work in this case. In order to
surmount this difficulty, we shall carefully use the energy
arguments. Precisely, we shall describe the oscillations caused by the  general initial data
  and include them in the energy estimates. Some ideas of this
type were introduced by Schochet \cite{S94} and extended to the case
of vanishing viscosity coefficients in \cite{M00,M01b}. To begin our
argument, we first give some formal analysis. Formally, by utilizing
the identity
$$
\nabla(|\tilde{\H}|^2)=2(\tilde{\H}\cdot \nabla
)\tilde{\H}+2\tilde{\H}\times \cu \tilde{\H},
$$
we can rewrite the momentum equation \eqref{a1b}
  as
\begin{equation}\label{a2a}
\partial_t (\tilde{\rho}\tilde{\u})+\dv(\tilde{\rho}\tilde{\u}\otimes\tilde{\u})
+\nabla \tilde{P} =(\tilde{\H}\cdot \nabla
  )\tilde{\H}-\frac12\nabla(|\tilde{\H}|^2)+\tilde{\mu}\Delta\tilde{\u}
  +(\tilde{\mu}+\tilde{\lambda})\nabla(\dv\tilde{\u}).
\end{equation}
By the identities
$$
\cu\cu \tilde{\H}=  \nabla\,\dv \tilde{\H}-\Delta \tilde{\H}$$ and
\begin{equation*}
 \cu( \tilde{ \u}\times\tilde{\H}) =
 \tilde{\u} (\dv \tilde{\H})  -  \tilde{\H}  (\dv\tilde{\u})
 + (\tilde{\H}\cdot \nabla)\tilde{\u} - (\tilde{\u}\cdot \nabla)\tilde{\H},
\end{equation*}
together with the constraint $\dv\tilde{\H}=0$,  the  magnetic field equation
\eqref{a1c} can be   expressed as
\begin{equation}\label{a2b}
 \partial_t\tilde{\H}   +
 (\dv \tilde{\u})\tilde{ \H}+ (\tilde{\u} \cdot \nabla)\tilde{\H}
 - (\tilde{\H}\cdot \nabla)\tilde{\u}= \tilde{\nu }  \Delta \tilde{\H}.
\end{equation}
 We  introduce the  scaling
\begin{equation*}
  \tilde{\rho} (x,  t  )=\rho^\epsilon (x, \epsilon t), \quad
 \tilde{\u} (x, t )=\epsilon \u^\epsilon(x,\epsilon t), \quad
  \tilde{\H} (x, t )=\epsilon \H^\epsilon(x,\epsilon t)
 \end{equation*}
and  assume that the viscosity coefficients $\tilde{\mu}$,
$\tilde{\lambda}$ and
 $\tilde{\nu}$ are  small constants and scaled like
\begin{equation}\label{pa}
  \tilde{\mu}=\epsilon \mu^\epsilon, \quad \tilde{\lambda}=\epsilon \lambda^\epsilon,
  \quad \tilde{\nu}=\epsilon \nu^\epsilon,
\end{equation}
where $\epsilon\in (0,1)$ is a small parameter and the normalized
coefficients $\mu^\epsilon$, $\lambda^\epsilon$   and $\nu^\epsilon$
satisfy  $\mu^\epsilon>0$, $2\mu^\epsilon+d\lambda^\epsilon\geq 0$,
and $\nu^\epsilon>0$.
 With the preceding scalings and using   the pressure  function \eqref{aad},  the compressible
MHD equations   \eqref{a1a}, \eqref{a2a} and \eqref{a2b} take the
form
\begin{align}
 & \partial_t {\rho}^\epsilon
 +\text{div}({\rho}^\epsilon{\u}^\epsilon)=0,  \label{a2i}\\
&  \partial_t
 ({\rho}^\epsilon{\u}^\epsilon)+\text{div}({\rho}^\epsilon{\u}^\epsilon\otimes
 {\u}^\epsilon)+\frac{a \nabla (\rho^\epsilon)^\gamma
 }{\epsilon^2}
  = ({\H}^\epsilon\cdot \nabla
  ){\H}^\epsilon-\frac12\nabla(|{\H}^\epsilon|^2)\nonumber\\
  &\qquad \qquad\qquad\qquad\qquad\qquad\qquad\qquad\quad
   +{\mu}^\epsilon\Delta{\u}^\epsilon+({\mu}^\epsilon+{\lambda}^\epsilon)
   \nabla(\dv {\u}^\epsilon),   \label{a2j}\\
 &  \partial_t {\H}^\epsilon   +
 (\dv  {\u}^\epsilon) { \H}^\epsilon+ ( {\u}^\epsilon \cdot \nabla) {\H}^\epsilon
 - ( {\H}^\epsilon\cdot \nabla) {\u}^\epsilon=  {\nu }^\epsilon
 \Delta  {\H}^\epsilon,\ \ \dv {\H}^\epsilon=0.\label{a2k}
 \end{align}
Formally if we let $\epsilon \rightarrow 0$, we obtain from the
momentum equation \eqref{a2j} that $\rho^\epsilon $ converges to
some function  $\bar \rho(t)\geq 0$ . If we further assume that the
initial datum $\rho^\epsilon_0$
 is of order $1+O(\epsilon)$ (this can be guaranteed by the
initial energy bound \eqref{bd} below), then we can expect that
$\bar\rho=1$. If the limits  ${\u}^\epsilon\rightarrow \u$ and
${\H}^\epsilon\rightarrow \H$ exist, then the continuity equation
\eqref{a2i} gives ${\rm div}\,{\u}=0$. Assuming  that
\begin{equation}\label{pam}
  \mu^\epsilon\rightarrow 0,    \quad   \nu^\epsilon \rightarrow 0 \quad \text{as}\quad
\epsilon \rightarrow 0,
\end{equation}
we obtain the following ideal incompressible MHD equations
 \begin{align}
& \partial_t \u+(\u\cdot \nabla)\u -({\H} \cdot \nabla
  ){\H}+\nabla p +\frac12\nabla(|{\H} |^2) =0, \label{a2l}\\
& \partial_t {\H}    + ( {\u}  \cdot \nabla) {\H}
 - ( {\H} \cdot \nabla) {\u} =  0, \label{a2m}\\
  &\dv \u=0, \quad   \dv \H=0.\label{a2n}
 \end{align}

In the present paper we will  prove rigorously   that a weak
solution of the compressible MHD
equations \eqref{a2i}-\eqref{a2k} with general initial data
 converges to, as  the  small parameter $\epsilon$
goes to $0$, the strong solution of the ideal incompressible MHD
equations \eqref{a2l}-\eqref{a2n} in the time interval where
the strong solution of \eqref{a2l}-\eqref{a2n}
exists.

\smallskip
  Before ending the introduction, we give the notations used throughout the present
  paper. We denote the space $L^q_2(\mathbb{R}^d)\,(q \geq 1)$ by
$$
L^q_2(\mathbb{R}^d)= \{f\in L_{loc}(\mathbb{R}^d):   f  1_{\{|f|\leq
1/2\}}\in L^2, f  1_{\{|f|\geq 1/2\}}\in L^q\} .
$$
The letters $C$ and $C_T$ denote various positive constants
independent of $\epsilon$, but $C_T$ may depend on $T$. For
convenience, we denote by $W^{s,r}\equiv W^{s,r}(\mathbb{R}^d)$ the
standard Sobolev space. In particular, if $r=2$ we denote
$W^{s,r}\equiv H^s$. For any vector field $\mathbf{v}$, we denote by
$P\mathbf{v}$ and $Q\mathbf{v}$, respectively, the divergence-free
part of $\mathbf{v}$ and the gradient part of $\mathbf{v}$, namely,
$Q\mathbf{v}=\nabla \Delta^{-1}(\text{div}\mathbf{v})$ and
$P\mathbf{v}=\mathbf{v}-Q\mathbf{v}$.

  We state our main results in Section 2 and present the proofs in Section 3.

\section{main results}

We first recall the local existence of strong solution to the ideal
incompressible MHD equations \eqref{a2l}-\eqref{a2n}. The proof can
be found in \cite{DL}.

\begin{prop}[\!\cite{DL}]\label{imhd}
Assume that the initial data $(\u(x,t),\H(x,t))|_{t=0} = (\u_0(x),
\linebreak  \H_0(x))$ satisfy $\u_0, \H_0\in
 {H}^s $, $s>d/2+1$, and $\dv \u_0=0, \dv \H_0=0$. Then, there exist a
 $T^*\in (0,\infty)$ and a unique solution
 $(\u,\H)\in L^{\infty}([0,T^*),{H}^s)$ to the ideal
incompressible MHD equations \eqref{a2l}-\eqref{a2n}
satisfying, for any $0<T<T^*$, $\dv \u =0$, $\dv\H =0$, and
\begin{equation}\label{ba}
\sup_{0\le t\le T}\!\big\{\|(\u,\H)(t)\|_{H^s}
+\|(\partial_t\u,\partial_t\H)(t)\|_{H^{s-1}}\big\} \le C_T.
\end{equation}
\end{prop}

We prescribe the initial conditions to the compressible
MHD equations   \eqref{a2i}-\eqref{a2k} as
\begin{equation}\label{bb}
  \rho^\epsilon|_{t=0}=\rho^\epsilon_0(x), \quad \rho^\epsilon
 \u^\epsilon|_{t=0}=\rho^\epsilon_0(x)\u^\epsilon_0(x)\equiv \mathbf{m}^\epsilon_0(x),
 \quad \H^\epsilon|_{t=0}=\H^\epsilon_0(x),
   \end{equation}
and assume that
\begin{equation}\label{bc}
 \rho_0^\epsilon \geq 0,\,\, \rho^\epsilon_0-1\in L^\gamma_2,\,\,
 \rho^\epsilon_0|\u^\epsilon_0|^2\in L^1,\,\,
 \mathbf{m}^\epsilon_0=0\,\, \text{for a.e.}\,\,
  {\rho^\epsilon_0=0} ,
\end{equation} and
\begin{equation}\label{bch}
\H^\epsilon_0\in L^2,\;\; \dv\H^\epsilon_0=0.
\end{equation}
 Moreover,  we assume that the initial data also satisfy the following uniform bound
  \begin{equation}\label{bd}
     \int_{\mathbb{R}^d}
 \Big[\frac12\rho^\epsilon_0|\u^\epsilon_0|^2+\frac12|\H^\epsilon_0|^2
 +\frac{a}{\epsilon^2{(\gamma-1)}}\big((\rho^\epsilon_0)^\gamma -1-\gamma
 (\rho^\epsilon_0-1)\big)
    \Big] dx \leq C.
  \end{equation}

Under the above assumptions, it was pointed out in \cite{HW1} that
the Cauchy problem of the compressible MHD
equations \eqref{a2i}-\eqref{a2k} has a   global   weak solution,
which can be stated as follows (see also \cite{HW3}).

\begin{prop}[\cite{HW1,HW3}]\label{cmhd}
 Let $\gamma>d/2$. Supposing that the initial data
 $(\rho^\epsilon_0,\mathbf{m}^\epsilon_0,\H^\epsilon_0 )$ satisfy the assumptions \eqref{bc}-\eqref{bd},
 then the compressible MHD equations
 \eqref{a2i}-\eqref{a2k} with  the initial data \eqref{bb} enjoy
  at least one global weak solution $(\rho^\epsilon,  \u^\epsilon, \H^\epsilon)$
 satisfying
\begin{enumerate}
  \item  $ \rho^\epsilon-1\in L^\infty(0,\infty;L^\gamma_2)\cap
 C([0,\infty),L^r_2)$ for all $1\leq r< \gamma$, $\rho^\epsilon
 |\u^\epsilon|^2\in
  L^\infty(0,\infty; L^1)$, $\H^\epsilon
\in   L^\infty(0,\infty; L^2)$, $\u^\epsilon\in L^2(0,T;H^1),$
  $\rho^\epsilon\u^\epsilon\in C([0,T],  L^{\frac{2\gamma}{\gamma+1}}_{weak})$,
  and  $\H^\epsilon \in  L^2(0,T;H^1)\cap C([0,T],L^{\frac{2\gamma}{\gamma+1}}_{weak})$
   for all $T\in (0,\infty)$;

  \item  the energy inequality
  \begin{equation}\label{be}
   \mathcal{E}^\epsilon(t)+\int^t_0  \mathcal{D}^\epsilon(s)ds\leq
 \mathcal{E}^\epsilon(0)
  \end{equation}
  holds with the finite total energy
  \begin{equation}\label{bf}
 \qquad\qquad  \mathcal{E}^\epsilon(t)\equiv \int_{\mathbb{R}^d}
 \bigg[\frac12\rho^\epsilon|\u^\epsilon|^2+\frac12|\H^\epsilon|^2
        +\frac{a}{\epsilon^2{(\gamma-1)}}\big((\rho^\epsilon)^\gamma
 -1-\gamma (\rho^\epsilon-1)\big)
    \bigg](t)
   \end{equation}
   and the dissipation energy
\begin{equation}\label{bff}
\mathcal{D}^\epsilon(t)\equiv \int_{\mathbb{R}^d}\big[ \mu^\epsilon|\nabla
 \u^\epsilon|^2
 +(\mu^\epsilon+\lambda^\epsilon) |{\rm div}\u^\epsilon|^2 + \nu^\epsilon |\nabla
 \H^\epsilon|^2\big](t);
\end{equation}

  \item the continuity equation is satisfied in the sense of
 renormalized solutions, i.e.,
  \begin{equation} \label{cnsbg}
    \partial_t b(\rho^\epsilon)+{ \rm div}
 (b(\rho^\epsilon)\u^\epsilon)+\big(b'(\rho^\epsilon)\rho^\epsilon-b(\rho^\epsilon)\big){\rm div}
 \u^\epsilon =0
  \end{equation}
  for any $b\in C^1(\mathbb{R})$ such that $b'(z)$ is a constant for
 $z$ large enough;
  \item the equations \eqref{a2i}-\eqref{a2k} hold in
 $\mathcal{D}'( \mathbb{R}^d\times(0,\infty))$.
\end{enumerate}
\end{prop}

The initial energy inequality \eqref{bd} implies that
$\rho^\epsilon_0$ is of order $1+O(\epsilon)$. We write
$\rho^\epsilon=1+\epsilon \varphi^\epsilon$ and denote
\begin{equation*}
  \Pi^\epsilon(x,t)=\frac{1}{\epsilon}\sqrt{\frac{2a}{\gamma-1}
  \big((\rho^\epsilon)^\gamma-1-\gamma(\rho^\epsilon-1)\big)}.
\end{equation*}
We will use the above approximation $\Pi^\epsilon(x,t)$ instead of
$\varphi^\epsilon$, since we can not obtain any bound for
$\varphi^\epsilon$ in $L^\infty(0,T;L^2)$ directly if $\gamma <2$.

 We also need to impose the following conditions  on the solution $(
\rho^\epsilon,\u^\epsilon,\H^\epsilon)$ at infinity
$$
 \rho^\epsilon \rightarrow 1 ,\quad
 \u^\epsilon \rightarrow \mathbf{0}, \ \ \H^\epsilon \rightarrow \mathbf{0} \quad \text{as}
 \quad |x|\rightarrow +\infty.
$$

 The main results of this paper can be stated as follows.

\begin{thm}\label{MRa} Let $s>{d}/{2}+2$ and the conditions in Proposition \ref{cmhd} hold.
Assume that the shear viscosity $\mu^\epsilon$
and the magnetic diffusion coefficient $\nu^\epsilon $ satisfy
\begin{align}\label{abc}
  \mu^\epsilon =\epsilon^\alpha,\quad \nu^\epsilon=\epsilon^\beta
\end{align}
for some constants $\alpha , \beta>0$ satisfying $0<\alpha+\beta
<2$.
  Moreover we
 assume that
$\sqrt{\rho^\epsilon_0}\u^\epsilon_0$ converges strongly in $L^2$ to
some $\tilde{\u}_0$, $\H^\epsilon_0$ converges strongly in $L^2$ to
some $\H_0$, and $\Pi^\epsilon_0$ converges strongly in $L^2$ to
some $\varphi_0$. Let $(\u,\H)$ be the smooth solution to the ideal
incompressible MHD equations \eqref{a2l}-\eqref{a2n}
defined on $[0,T^*)$ with $(\u,\H)|_{t=0}=(\u_0,\H_0), \u_0=P\tilde{
\u}_0$. Then, for any $0<T<T^*$, the global weak solution
$(\rho^\epsilon, \u^\epsilon, \H^\epsilon)$ of the compressible
MHD equations \eqref{a2i}-\eqref{a2k} established in
Proposition \ref{cmhd} satisfies
\begin{enumerate}
          \item  $\rho^\epsilon$ converges strongly to $1$ in
 $C([0,T],L^\gamma_2(\mathbb{R}^d))$;
           \item  $\H^\epsilon$ converges strongly to $  \H$
  in $L^\infty (0,T;L^2(\mathbb{R}^d))$;
          \item $P(\sqrt{\rho^\epsilon}\u^\epsilon)$ converges  strongly
 to $\u$ in $L^\infty (0,T;L^2(\mathbb{R}^d))$;
           \item $\sqrt{\rho^\epsilon}\u^\epsilon$ converges strongly to
 $\u$ in $L^r(0,T; L^2_{\mathrm{loc}}(\mathbb{R}^d))$ for all $1\leq r<+\infty$.
          \end{enumerate}
\end{thm}

The proof of above results is based on  the combination  of  the
modulated energy method, motivated by Brenier \cite{B00}, the
Strichartz's estimate of linear wave equations \cite{DG99} and the
weak convergence method. Masmoudi \cite{M01b} made use of such ideas
to consider  the incompressible, inviscid convergence of weak
solutions of compressible Navier-Stokes equations to the strong
solution of the incompressible Euler equations in the whole space
and the torus. We shall follow and adapt the method used in
\cite{M01b} to prove our results. Besides the difficulties pointed
out in \cite{M01b}, we have to overcome the additional difficulties
caused by the strong coupling of the hydrodynamic motion and  the
magnetic fields. Very careful energy estimates are employed to
estimate those nonlinear terms for large data, see Step 4 in the
proof of Theorem \ref{MRa} below. On the other hand, because the
general initial data are considered in the present paper, the
oscillations appear and propagate with the solution. We will use the
Strichartz's estimate of linear wave equations to deal with such
oscillations and their interactions with velocity and magnetic
fields. Finally, the weak convergence method and refined energy
analysis are employed to obtain the desired convergence results.

\begin{rem}
The assumption that $\Pi^\epsilon_0$ converges strongly in $L^2$ to
some $\varphi_0$ in fact implies that $\varphi^\epsilon_0$ converges
strongly  to $\varphi_0$ in $L^\gamma_2$.
\end{rem}

\begin{rem}
The condition \eqref{abc} is required in our proof to control the
terms caused by  the strong coupling between the hydrodynamic motion
and  the magnetic fields,   see \eqref{rr3} below on the estimates
of $S^\epsilon_4(t)$, where $\rho^\epsilon, \u^\epsilon$ and
$\H^\epsilon$ are involved. Notice that such condition is not needed
in the proof of  the inviscid and incompressible limit for the
compressible Navier-Stokes equations, see \cite{M01b}. Thus it shows
the presence of magnetic effect for our problem.

\end{rem}

\begin{rem}
Compared with the results obtained in \cite{JJL} on the periodic
case where the convergence of compressible MHD equations to partial
viscous MHD equations (shear viscosity coefficient is zero but
magnetic diffusion coefficient is a positive constant)   is
rigorously proved, here we can allow the  magnetic diffusion
coefficient also goes to zero due to the dispersive property of the
acoustic wave in the whole space.
\end{rem}

\begin{rem}
  Our arguments in this paper can be applied, after slight
modifications,
 to the case
 \begin{equation*}
  \mu^\epsilon\rightarrow \mu^0,    \quad   \nu^\epsilon \rightarrow \nu^0 \quad \text{as}\quad
\epsilon \rightarrow 0,
\end{equation*}
 where $\mu^0 $ and $ \nu^0$ are nonnegative constants satisfying
(1) $\mu_0>0, \nu^0=0$, or (2) $\mu_0=0, \nu^0>0$ or  (3) $\mu_0>0,
\nu^0>0$.  Hence we can also obtain the convergence of
  compressible MHD equations \eqref{a2i}-\eqref{a2k} to the incompressible MHD
  equations
\begin{align*}
& \partial_t \u+(\u\cdot \nabla)\u -({\H} \cdot \nabla
  ){\H} -\mu^0 \Delta \u +\nabla p +\frac12\nabla(|{\H} |^2) =0,  \\
& \partial_t {\H}    + ( {\u}  \cdot \nabla) {\H}
 - ( {\H} \cdot \nabla) {\u} -\nu^0 \Delta \H =  0, \\
  &\dv \u=0, \quad   \dv \H=0.
 \end{align*}
We omit the details for conciseness.
\end{rem}

\section{  Proof of Theorem \ref{MRa}}

In this section, we shall prove our convergence results by combining the
modulated energy method with the Strichartz's estimate of linear wave
equations and  employing  the weak convergence method.

\begin{proof}[Proof of Theorem \ref{MRa}]
  We divide the proof into several steps.

\medskip
\emph{Step 1: Basic energy estimates and compact arguments.}

By the assumption on the initial data  we obtain, from the energy
inequality \eqref{be}, that the total energy
$\mathcal{E}^\epsilon(t)$ has a uniform upper bound  for a.e. $t\in
[0,T], T>0$. This uniform bound
implies that $\rho^\epsilon |\u^\epsilon|^2$ and
$ \big((\rho^\epsilon)^\gamma-1-\gamma(\rho^\epsilon-1)\big)/{\epsilon^2}$
are bounded in $L^\infty(0,T;L^1)$ and $ \H^\epsilon $ is
bounded in $L^\infty(0,T;L^2)$.
By an analysis similar to that used in \cite{LM98},  we find
\begin{equation}\label{L2g}
\int_{\mathbb{R}^d}\frac{1}{\epsilon^2}|\rho^\epsilon-1|^2
1_{\{|\rho^\epsilon-1|\leq \frac12\}}
  +\int_{\mathbb{R}^d}\frac{1}{\epsilon^2}|\rho^\epsilon-1|^\gamma
  1_{\{|\rho^\epsilon-1|\geq \frac12\}}   \leq C,
\end{equation}
which implies   that
\begin{equation}\label{re}
    \rho^\epsilon \rightarrow    1 \ \
    \text{strongly in}\ \
C([0,T],L^\gamma_2).
\end{equation}
By the results of  \cite{DG99}, we have the following estimate on
$\u^\epsilon$,
\begin{align}\label{uepsilon}
\|\u^\epsilon\|^2_{L^2}\leq C+C\epsilon^{4/d}\|\nabla
\u^\epsilon\|^2_{L^2},\quad  d=2,3.
\end{align}
Furthermore, the fact that $\rho^\epsilon |\u^\epsilon|^2$ and $
|\H^\epsilon|^2$ are bounded in  $L^\infty(0,T;L^1)$ implies the
following convergence (up to the extraction of a subsequence
$\epsilon_n$):
\begin{align*}
& \sqrt{\rho^\epsilon}  \u^\epsilon   \  \text{converges weakly-$\ast$}
\,\, \text{to some}\, \mathbf{J} \, \,\text{in}\,
L^\infty(0,T;L^2),\\
&  \H^\epsilon \    \text{converges weakly-$\ast$}
\,\, \text{to some}\, \mathbf{K} \, \,\text{in}\,
L^\infty(0,T;L^2).
\end{align*}

Our main task in this section is to show that $\mathbf{J}=\u$ and  $\mathbf{K}=\H$ in some sense,
where $(\u,\H)$  is the strong solution to the ideal
incompressible MHD equations \eqref{a2l}-\eqref{a2n}.

\medskip

\emph{Step 2: Description and cancelation of the oscillations. }

In order to describe the oscillations caused by the initial data,
we use the ideas introduced in \cite{S94,M01b} and the
dispersion property of the linear wave equations \cite{DG99,M01b}.

We project the momentum equation  \eqref{a2j} on the
``gradient vector-fields'' to  find
\begin{align}\label{Pr}
&\partial_t Q(\rho^\epsilon\u^\epsilon)+Q[\dv(\rho^\epsilon\u^\epsilon\otimes\u^\epsilon)]
  - (2\mu^\epsilon+\lambda^\epsilon)\nabla\dv\u^\epsilon+\frac12\nabla(|{\H}^\epsilon|^2)\nonumber\\
 &\quad -Q[({\H}^\epsilon\cdot \nabla
  ){\H}^\epsilon] +\frac{a}{\epsilon^2}\nabla \big((\rho^\epsilon)^\gamma
 -1-\gamma (\rho^\epsilon-1)\big)+\frac{a\gamma}{\epsilon^2}\nabla
 (\rho^\epsilon-1)=0.
\end{align}
Below  we    assume $a\gamma=1$ for simplicity, otherwise we can
change $\varphi^\epsilon$ to $ \varphi^\epsilon/({a\gamma}).$
Noticing $\rho^\epsilon=1+\epsilon\varphi^\epsilon $,  we can write
\eqref{a2i} and \eqref{Pr} as
\begin{align*}
&\epsilon\partial_t \varphi^\epsilon+\dv Q(\rho^\epsilon \u^\epsilon)=0,\\
 &\epsilon\partial_t  Q(\rho^\epsilon \u^\epsilon)+ \nabla \varphi^\epsilon=\epsilon\mathbf{F}^\epsilon,
\end{align*}
where $ \mathbf{F}^\epsilon$ is given by
\begin{align*}
\mathbf{F}^\epsilon=&-Q[\dv(\rho^\epsilon\u^\epsilon\otimes\u^\epsilon)]
  +(2\mu^\epsilon+\lambda^\epsilon)\nabla\dv\u^\epsilon-\frac12\nabla(|{\H}^\epsilon|^2)\\
  & +Q[({\H}^\epsilon\cdot \nabla
  ){\H}^\epsilon] -\frac{1}{\gamma\epsilon^2}\nabla \big((\rho^\epsilon)^\gamma
 -1-\gamma (\rho^\epsilon-1)\big).
\end{align*}
Therefore, we introduce the following group defined by
$\mathcal{L}(\tau)=e^{\tau L}$, $\tau \in \mathbb{R}$, where $L$ is
the operator defined on $\mathcal{D}'\times (\mathcal{D}')^{d}$ by
$$
L\Big(\begin{array}{c}
 \phi   \\
\mathbf{v}
\end{array}
\Big) =\Big(\begin{array}{c}
 -\text{div} \mathbf{v}  \\
-\nabla \phi
\end{array}
\Big).
$$
Then, it is easy to check that $e^{\tau L}$ is an isometry on each
 $H^r\times (H^r)^d$ for all $r\in \mathbb{R}$ and for all $\tau \in
 \mathbb{R}$.
Denoting
$$
\Big(\begin{array}{c}
 \bar \phi (\tau)   \\
\bar{\mathbf{v}}(\tau)
\end{array}
\Big) =e^{\tau L}\Big(\begin{array}{c}
 \phi  \\
\mathbf{v}
\end{array}
\Big),
$$
then we have
$$
\frac{\partial \bar \phi}{\partial \tau}=-\text{div} \bar{\mathbf{v}}, \quad
 \frac{\partial \bar{\mathbf{v}}}{\partial \tau}=-\nabla   \bar \phi.
$$
Thus, $\frac{\partial^2\bar \phi}{\partial \tau ^2}-\Delta \bar \phi
 =0$.

Let $(\phi^\epsilon, \mathbf{g}^\epsilon=\nabla q^\epsilon)$ be the solution
of the following system
\begin{align}
   &   \frac{\partial \phi^\epsilon}{\partial
 t}=-\frac{1}{\epsilon}\text{div} \g^\epsilon, \qquad   \phi^\epsilon|_{t=0} =\Pi^\epsilon_0,
 \label{ca} \\
   &  \frac{\partial \g^\epsilon}{\partial t}=-\frac{1}{\epsilon}\nabla
 \phi^\epsilon,  \qquad \ \ \
 \g^\epsilon|_{t=0} =Q(\sqrt{\rho^\epsilon_0}\u^\epsilon_0).\label{cb}
 \end{align}
Here we have used $ Q(\sqrt{\rho^\epsilon_0}\u^\epsilon_0)$ as an
approximation of $Q( {\rho^\epsilon_0}\u^\epsilon_0)$ since
$\|Q({\rho^\epsilon_0}\u^\epsilon_0)-Q(\sqrt{\rho^\epsilon_0}\u^\epsilon_0)\|_{L^1}\rightarrow
0$ as $\epsilon \rightarrow 0$.

Our main idea is to use $ \phi^\epsilon$ and $\g^\epsilon$ as test
functions and plug them into the total energy
$\mathcal{E}^\epsilon(t)$ to cancel the oscillations. We introduce
the following regularization  for the  initial data,
$\Pi^{\epsilon,\delta}_0=\Pi^\epsilon_0 \ast \chi^\delta$,
$Q^\delta( \sqrt{\rho^\epsilon_0}\u^\epsilon_0)=Q(
\sqrt{\rho^\epsilon_0}\u^\epsilon_0) \ast \chi^\delta$, and denote
by $(\phi^{\epsilon,\delta},\g^{\epsilon,\delta}=\nabla
q^{\epsilon,\delta})$ the corresponding solution to the equations
 \eqref{ca}-\eqref{cb}
with initial data
$\phi^{\epsilon,\delta}|_{t=0}=\Pi^{\epsilon,\delta}_0$,
$\g^{\epsilon,\delta}|_{t=0}=Q^\delta(
\sqrt{\rho^\epsilon_0}u^\epsilon_0)$. Here $\chi \in
C^\infty_0(\mathbb{R}^d)$ is the Friedrich's mollifier, i.e., $\int_{\mathbb{R}^d}
\chi =1$ and $\chi^\delta(x)=(1/\delta^d)\chi(x/\delta)$. Since the
equations \eqref{ca}-\eqref{cb} are linear, it is easy to
verify that $\phi^{\epsilon,\delta}=\phi^{\epsilon}\ast \chi^\delta ,
\g^{\epsilon,\delta}=\g^{\epsilon}\ast \chi^\delta$.

Using the Strichartz estimate of the linear wave
 equations \cite{DG99}, we have
\begin{equation}\label{cc}
\bigg\|\Big(\begin{array}{c}
  \phi^{\epsilon,\delta}     \\
\nabla q^{\epsilon,\delta}
\end{array}
\Big)\bigg\|_{L^l(\mathbb{R}, W^{s,r})} \leq
C\epsilon^{1/l}\bigg\|\bigg(\begin{array}{c}
 \Pi^{\epsilon,\delta}_0    \\
Q^\delta( \sqrt{\rho^\epsilon_0}\u^\epsilon_0)
\end{array}
\Big)\bigg\|_{H^{s+\sigma}}
\end{equation}
for all $l,r> 2$ and $\sigma>0$ such that
$$
\frac{2}{r}=(d-1)\Big(\frac12-\frac{1}{l}\Big), \qquad \sigma
=\frac{d+1}{d-1}.
$$
The estimate \eqref{cc} implies that for arbitrary but fixed $\delta$ and
for all $s\in \mathbb{R}$, we have
\begin{equation} \label{cd}
  \phi^{\epsilon,\delta},\,\, \g^{\epsilon,\delta}\rightarrow 0 \quad
 \text{in} \ \ L^l(\mathbb{R}, W^{s,r} ) \ \ \text{as}\ \
\epsilon \rightarrow 0.
\end{equation}

%

\medskip
\emph{Step 3: The modulated energy functional and uniform estimates.}

We first recall the energy inequality of the compressible
MHD equations \eqref{a2i}-\eqref{a2k} for almost all
$t$,
 \begin{align}\label{ce}
 & \frac12\int_{\mathbb{R}^d}\Big[\rho^\epsilon(t)|\u^\epsilon|^2(t)+|\H^\epsilon|^2(t)
    +(\Pi^\epsilon(t))^2
      \Big]
+\mu^\epsilon \int^t_0\! \int_{\mathbb{R}^d}|\nabla \u^\epsilon|^2 \nonumber\\
&\ \  +(\mu^\epsilon+\lambda^\epsilon)\int^t_0\!
\int_{\mathbb{R}^d}|{\rm div}\u^\epsilon |^2
+\nu^\epsilon\int^t_0\! \int_{\mathbb{R}^d}|\nabla \H^\epsilon|^2\nonumber\\
& \ \ \   \leq  \frac12\int_{\mathbb{R}^d}\Big[\rho_0^\epsilon|\u_0^\epsilon|^2
+|\H_0^\epsilon|^2+(\Pi^\epsilon_0)^2     \Big] .
 \end{align}
The conservation of energy for the ideal incompressible MHD equations
\eqref{a2l}-\eqref{a2n} reads
\begin{equation}\label{cf}
  \frac12 \int_{\mathbb{R}^d}\big[|\u|^2(t)+ |\H|^2(t)\big]   =\frac12  \int_{\mathbb{R}^d}
  \big[|\u_0|^2+|\H_0|^2\big].
\end{equation}
From the system \eqref{ca}-\eqref{cb} we get that, for all $t$,
\begin{equation}\label{cg}
\frac12 \int_{\mathbb{R}^d}
\Big[|\phi^{\epsilon,\delta}|^2+|\g^{\epsilon,\delta}|^2\Big](t)=
\frac12 \int_{\mathbb{R}^d}
\Big[|\phi^{\epsilon,\delta}|^2+|\g^{\epsilon,\delta}|^2\Big](0).
\end{equation}
Using $\phi^{\epsilon,\delta}$ as a test function and noticing
$\rho^\epsilon=1+\epsilon \varphi^\epsilon$, we obtain the following
weak formulation of the continuity equation \eqref{a2i}
\begin{equation}\label{ch}
  \int_{\mathbb{R}^d}
 \phi^{\epsilon,\delta}(t)\varphi^\epsilon(t)+\frac{1}{\epsilon}\int^t_0\int_{\mathbb{R}^d}
\big[\text{div}(\nabla
q^{\epsilon,\delta})\varphi^\epsilon+\mbox{div}(\rho^\epsilon
\u^\epsilon) \phi^{\epsilon,\delta}\big] =\int_{\mathbb{R}^d}
\phi^{\epsilon,\delta}(0)\varphi^\epsilon_0.
\end{equation}
We use $\u$ and $\g^{\epsilon,\delta}=\nabla q^{\epsilon,\delta}$ as
  test functions to  the momentum equation \eqref{a2j}
respectively  to deduce
\begin{align}\label{ci}
&  \int_{\mathbb{R}^d}(\rho^\epsilon \u^\epsilon \cdot \u)(t)
+\int^t_0\int_{\mathbb{R}^d}\rho^\epsilon \u^\epsilon \cdot \big[(\u\cdot \nabla)
\u-(\H\cdot \nabla)\H+\nabla p+\frac12\nabla(|\H|^2)\big] \nonumber\\
&  -\int^t_0\int_{\mathbb{R}^d}\big[(\rho^\epsilon \u^\epsilon \otimes
\u^\epsilon)\cdot \nabla \u+   (\H^\epsilon\cdot \nabla)\H^\epsilon \cdot \u
 - \mu^\epsilon  \nabla  \u^\epsilon  \cdot \nabla \u\big] = \int_{\mathbb{R}^d}
  \rho^\epsilon_0 \u^\epsilon_0 \cdot \u_0,
\end{align}
and
\begin{align}\label{cj}
&  \int_{\mathbb{R}^d}(\rho^\epsilon \u^\epsilon \cdot \nabla
q^{\epsilon,\delta})(t) +\int^t_0\int_{\mathbb{R}^d}\rho^\epsilon \u^\epsilon
\cdot\Big(\frac{1}{\epsilon}\nabla
 \phi^{\epsilon,\delta}\Big)
-\int^t_0\int_{\mathbb{R}^d}(\rho^\epsilon \u^\epsilon \otimes \u^\epsilon)\cdot
\nabla
 \g^{\epsilon,\delta}\nonumber\\
&\quad  +\int^t_0\int_{\mathbb{R}^d}\big[\mu^\epsilon \nabla \u^\epsilon \cdot \nabla
\g^{\epsilon,\delta} +(\mu^\epsilon+\lambda^\epsilon)\text{div}\u^\epsilon \text{div}\,
 \g^{\epsilon,\delta}\big]\nonumber\\
& \quad - \int^t_0\int_{\mathbb{R}^d}(\H^\epsilon\cdot \nabla)\H^\epsilon \cdot \g^{\epsilon,\delta}
 -\int^t_0\int_{\mathbb{R}^d}\frac12|\H^\epsilon|^2\dv\g^{\epsilon,\delta}\nonumber\\
& \quad
 -\int^t_0\int_{\mathbb{R}^d}
\Big(\frac{1}{\epsilon}\varphi^\epsilon+\frac{\gamma-1}{2}(\Pi^\epsilon)^2\Big)\text{div}\,
\g^{\epsilon,\delta}  = \int_{\mathbb{R}^d}\rho^\epsilon_0 \u^\epsilon_0 \cdot
\g^{\epsilon,\delta}(0).
\end{align}
Similarly, using  $\H$ as a test function to the magnetic field
equation \eqref{a2k}, we get
\begin{align}\label{cj1}
&  \int_{\mathbb{R}^d}(\H^\epsilon \cdot \H)(t)
+\int^t_0\int_{\mathbb{R}^d}\H^\epsilon \cdot\big[(\u\cdot \nabla)
\H-(\H\cdot \nabla)\u\big] + \nu^\epsilon
\int^t_0\int_{\mathbb{R}^d}
\nabla  \H^\epsilon  \cdot \nabla \H\nonumber\\
&\quad +\int^t_0\int_{\mathbb{R}^d}\big[(\dv  {\u}^\epsilon) { \H}^\epsilon
+ ( {\u}^\epsilon \cdot \nabla) {\H}^\epsilon  - ( {\H}^\epsilon\cdot \nabla)
 {\u}^\epsilon\big]\cdot \H  = \int_{\mathbb{R}^d} \H^\epsilon_0 \cdot \H_0.
\end{align}
Summing up \eqref{ce}, \eqref{cf} and \eqref{cg}, and inserting
\eqref{ch}-\eqref{cj1} into the resulting inequality,
 we can deduce the following
inequality by a straightforward computation
\begin{align}\label{ck}
&   \frac{1}{2}\int_{\mathbb{R}^d}\Big\{  |\sqrt{\rho^\epsilon} \u^\epsilon
-\u-\g^{\epsilon,\delta}|^2(t)+ |  \H^\epsilon
-\H|^2(t)
  +(\Pi^\epsilon-\phi^{\epsilon,\delta})^2(t)\Big\}\nonumber\\
&\quad
+ {\mu^\epsilon}  \int^t_0\int_{\mathbb{R}^d}|\nabla \u^\epsilon|^2
+ (\mu^\epsilon+\lambda^\epsilon)\int^t_0\! \int_{\mathbb{R}^d}|{\rm div}
\u^\epsilon |^2+ {\nu^\epsilon} \int^t_0\int_{\mathbb{R}^d}|\nabla \H^\epsilon |^2\nonumber\\
& \leq \frac{1}{2}\int_{\mathbb{R}^d}\Big\{ |\sqrt{\rho^\epsilon}
\u^\epsilon -\u-\g^{\epsilon,\delta}|^2(0)+|\H^\epsilon -\H |^2(0)
+(\Pi^\epsilon_0-\phi^{\epsilon,\delta}(0))^2\Big\}\nonumber\\
&\quad+{\mu^\epsilon}  \int^t_0\int_{\mathbb{R}^d}\nabla
\u^\epsilon\cdot\nabla \u +       {\nu^\epsilon}
\int^t_0\int_{\mathbb{R}^d}\nabla \H^\epsilon\cdot\nabla \H
                  +R_1^\epsilon(t)+R_2^\epsilon(t)+ R_3^{\epsilon,\delta}(t),
\end{align}
where
\begin{align*}
R_1^\epsilon(t)=& -\int^t_0\int_{\mathbb{R}^d}\rho^\epsilon
\u^\epsilon\cdot [(\H\cdot \nabla) \H]
    -\int^t_0\int_{\mathbb{R}^d}
(\H^\epsilon\cdot \nabla) \H^\epsilon \cdot\u \nonumber\\
& +\int^t_0\int_{\mathbb{R}^d}\H^\epsilon \cdot\big[(\u\cdot \nabla)
\H-(\H\cdot
\nabla)\u\big]+\frac{1}{2}\int^t_0\int_{\mathbb{R}^d}\rho^\epsilon
\u^\epsilon\cdot \nabla(|\H|^2)\nonumber\\
&
 +\int^t_0\int_{\mathbb{R}^d}\big[(\dv  {\u}^\epsilon) { \H}^\epsilon
 + ( {\u}^\epsilon \cdot \nabla) {\H}^\epsilon
 - ( {\H}^\epsilon\cdot \nabla) {\u}^\epsilon\big]\cdot
 \H,\nonumber\\
R_2^\epsilon(t)=&\int^t_0\int_{\mathbb{R}^d}\rho^\epsilon
\u^\epsilon \cdot\big[ (\u\cdot \nabla) \u +\nabla p
\big]-\int^t_0\int_{\mathbb{R}^d}(\rho^\epsilon \u^\epsilon\otimes
\u^\epsilon)\cdot \nabla u,\nonumber\\
R_3^{\epsilon,\delta}(t)=&\int_{\mathbb{R}^d}
\sqrt{\rho^\epsilon}-1)\sqrt{\rho^\epsilon}\u^\epsilon\cdot
\g^{\epsilon,\delta}(t)-\int_{\mathbb{R}^d}
\sqrt{\rho^\epsilon}-1)\sqrt{\rho^\epsilon}\u^\epsilon\cdot
\g^{\epsilon,\delta}(0) \nonumber\\
&\int_{\mathbb{R}^d}
(\sqrt{\rho^\epsilon}-1)\sqrt{\rho^\epsilon}\u^\epsilon\cdot
\g^{\epsilon,\delta}(t)-\int_{\mathbb{R}^d}
\big[(\Pi^\epsilon-\varphi^\epsilon)\phi^{\epsilon,\delta}\big](t)\nonumber\\
&   -\int^t_0\int_{\mathbb{R}^d}(\rho^\epsilon \u^\epsilon
\otimes \u^\epsilon)\cdot \nabla \g^{\epsilon,\delta}\nonumber
+\int^t_0\int_{\mathbb{R}^d} \big[\mu^\epsilon \nabla \u^\epsilon
\cdot \nabla \g^{\epsilon,\delta}
+(\mu^\epsilon+\lambda^\epsilon)\text{div}\u^\epsilon \text{div}
\g^{\epsilon,\delta}\big]\nonumber\\
  & - \frac{\gamma-1}{2}\int^t_0\int_{\mathbb{R}^d}
(\Pi^\epsilon)^2 \text{div} \g^{\epsilon,\delta}
-\int_{\mathbb{R}^d}
(\sqrt{\rho^\epsilon}-1)\sqrt{\rho^\epsilon}\u^\epsilon\cdot
\g^{\epsilon,\delta}(0)\nonumber\\
&
-\int^t_0\int_{\mathbb{R}^d}\frac12|\H^\epsilon|^2\dv\g^{\epsilon,\delta}
 - \int^t_0\int_{\mathbb{R}^d}(\H^\epsilon\cdot \nabla)\H^\epsilon \cdot \g^{\epsilon,\delta}
+\int_{\mathbb{R}^d}(\Pi^\epsilon_0-\varphi^\epsilon_0)\phi^{\epsilon,\delta}(0).\nonumber
\end{align*}

We first deal with $R_1^\epsilon(t)$ and $R_2^\epsilon(t)$ on the
right-hand side of the inequality \eqref{ck}. Denoting
$\mathbf{w}^{\epsilon,\delta}=\sqrt{\rho^\epsilon} \u^\epsilon
-\u-\g^{\epsilon,\delta}$,   $\mathbf{Z}^{\epsilon}= \H^\epsilon
-\H$,  integrating by parts and making use of the facts that ${\rm
div}\, \H^\epsilon =0, {\rm div}\, \u =0,$ and ${\rm div}\,\H =0$,
we obtain that
\begin{align}\label{ck1}
R_1^\epsilon(t)
 = &  -\int^t_0\int_{\mathbb{R}^d}\rho^\epsilon
\u^\epsilon\cdot [(\H\cdot \nabla) \H] +\int^t_0\int_{\mathbb{R}^d}
(\H^\epsilon\cdot \nabla) \u \cdot \H^\epsilon \nonumber\\
& +\int^t_0\int_{\mathbb{R}^d}(\u\cdot \nabla)\H\cdot \H^\epsilon
-\int^t_0\int_{\mathbb{R}^d}(\H\cdot \nabla)\u  \cdot \H^\epsilon  \nonumber\\
&  - \int^t_0\int_{\mathbb{R}^d}( {\u}^\epsilon \cdot \nabla)\H \cdot {\H}^\epsilon
 + \int^t_0\int_{\mathbb{R}^d}( {\H}^\epsilon\cdot \nabla) \H \cdot {\u}^\epsilon
 +\frac{1}{2}\int^t_0\int_{\mathbb{R}^d}\rho^\epsilon
\u^\epsilon\cdot \nabla(|\H|^2)  \nonumber\\
 = & \int^t_0\int_{\mathbb{R}^d}(1-\rho^\epsilon)
\u^\epsilon\cdot [(\H\cdot \nabla) \H]
+\int^t_0\int_{\mathbb{R}^d}[(\H^\epsilon-\H)\cdot \nabla]\u  \cdot (\H^\epsilon-\H)\nonumber\\
&+\int^t_0\int_{\mathbb{R}^d}[(\H^\epsilon-\H)\cdot \nabla]\H  \cdot (\u^\epsilon-\u)
- \int^t_0\int_{\mathbb{R}^d}[(\u^\epsilon-\u)\cdot \nabla]\H  \cdot (\H^\epsilon-\H)\nonumber\\
& +\frac12\int^t_0\int_{\mathbb{R}^d}(\rho^\epsilon-1){\u}^\epsilon \nabla(| \H|^2)\nonumber\\
\leq & \int^t_0\int_{\mathbb{R}^d}(1-\rho^\epsilon)
\u^\epsilon\cdot [(\H\cdot \nabla) \H)]
+\int^t_0    \|\mathbf{Z}^\epsilon(s)\|^2_{L^2}\| \nabla \u(s)\|_{L^\infty}  ds\nonumber\\
&+  \int^t_0   \big[\|\mathbf{w}^{\epsilon,\delta}(s)\|^2_{L^2} +
\|\mathbf{Z}^\epsilon(s)\|^2_{L^2}\big] \| \nabla \H(s)\|_{L^\infty}
ds
\nonumber\\
&+\int^t_0\!\int_{\mathbb{R}^d}(\mathbf{Z}^\epsilon \cdot \nabla)\H
\cdot [(1-\sqrt{\rho^\epsilon})\u^\epsilon+\g^{\epsilon,\delta}]
\nonumber\\
&- \int^t_0\!\int_{\mathbb{R}^d}\big\{[(1-\sqrt{\rho^\epsilon})\u^\epsilon
+\g^{\epsilon,\delta}]\cdot \nabla\big\}\H  \cdot   \mathbf{Z}^\epsilon
+\frac12\int^t_0\int_{\mathbb{R}^d}(\rho^\epsilon-1){\u}^\epsilon \nabla(|\H|^2)
\end{align}
and
\begin{align}\label{cl}
 R_2^\epsilon(t)
    =& -\int^t_0\int_{\mathbb{R}^d}(\mathbf{w}^{\epsilon,\delta}\otimes
 \mathbf{w}^{\epsilon,\delta})\cdot \nabla \u
       +\int^t_0\int_{\mathbb{R}^d}(\rho^\epsilon-\sqrt{\rho^\epsilon})\u^\epsilon\cdot
 ((\u\cdot \nabla) \u)\nonumber\\
   & -\int^t_0\int_{\mathbb{R}^d}(\g^{\epsilon,\delta}\cdot \nabla) \u\cdot
 \mathbf{w}^{\epsilon,\delta}
    +\int^t_0\int_{\mathbb{R}^d}[(\sqrt{\rho^\epsilon}\u^\epsilon-\u)\cdot \nabla] \u
 \cdot \g^{\epsilon,\delta}\nonumber\\
   &     + \int^t_0 \int_{\mathbb{R}^d}\rho^\epsilon \u^\epsilon \cdot  \nabla p
    -\int^t_0\int_{\mathbb{R}^d}(\sqrt{ \rho^\epsilon}\u^\epsilon- \u) \cdot \nabla
 \big(\frac{|\u|^2}{2}\big).
    \end{align}

Substituting \eqref{ck1} and \eqref{cl} into the inequality \eqref{ck}, we conclude
\begin{align}\label{cm}
 &
 \|\mathbf{w}^{\epsilon,\delta}(t)\|^2_{L^2}+
 \|\mathbf{Z}^{\epsilon}(t)\|^2_{L^2}+\|\Pi^\epsilon(t)-\phi^{\epsilon,\delta}(t)\|^2_{L^2}
  + 2{\mu^\epsilon} \int^t_0\int_{\mathbb{R}^d}|\nabla \u^\epsilon|^2 \nonumber\\
& + 2(\mu^\epsilon+\lambda^\epsilon)\int^t_0\!
\int_{\mathbb{R}^d}|{\rm div}\u^\epsilon |^2 + 2{\nu^\epsilon}
\int^t_0\int_{\mathbb{R}^d}|\nabla \H^\epsilon|^2
\nonumber\\
\leq &
\|\mathbf{w}^{\epsilon,\delta}(0)\|^2_{L^2}+\|\mathbf{Z}^{\epsilon}(0)\|^2_{L^2}
+\|\Pi^\epsilon_0-\phi^{\epsilon,\delta}(0)\|^2_{L^2}\nonumber\\
& +2C \int^t_0 \big(\|\mathbf{w}^{\epsilon,\delta}(s)\|^2_{L^2}+
 \|\mathbf{Z}^{\epsilon}(s)\|^2_{L^2}\big)\big(\|\nabla \u(s)\|_{L^\infty}+\|\nabla
\H(s)\|_{L^\infty}\big)ds\nonumber\\
&   +2S^{\epsilon,\delta}_1(t)+ 2\sum_{i=2}^7S^\epsilon_2(t),
\end{align}
where
\begin{align*}
S^{\epsilon,\delta}_1(t)=   &    R^{\epsilon,\delta}_1(t)
+\int^t_0\!\int_{\mathbb{R}^d}(\mathbf{Z}^\epsilon \cdot \nabla)
\H  \cdot\g^{\epsilon,\delta}- \int^t_0\!\int_{\mathbb{R}^d} (\g^{\epsilon,\delta}\cdot \nabla)\H
\cdot \mathbf{Z}^\epsilon\\
& -\int^t_0\int_{\mathbb{R}^d}(\g^{\epsilon,\delta}\cdot
\nabla) \u\,
  \mathbf{w}^{\epsilon,\delta}+\int^t_0\int_{\mathbb{R}^d}[(\sqrt{\rho^\epsilon}\u^\epsilon-\u)
    \cdot \nabla] \u \cdot\g^{\epsilon,\delta},\nonumber\\
 S^{\epsilon}_2(t)= & {\mu^\epsilon}
\int^t_0\int_{\mathbb{R}^d}\nabla \u^\epsilon\cdot\nabla \u
+       {\nu^\epsilon}  \int^t_0\int_{\mathbb{R}^d}\nabla \H^\epsilon\cdot\nabla \H, \\
S^{\epsilon}_3(t)=&\int_{\mathbb{R}^d}
\big[(\sqrt{\rho^\epsilon}-1)\sqrt{\rho^\epsilon}\u^\epsilon\cdot
 \u \big](t)  -\int_{\mathbb{R}^d}
\big[(\sqrt{\rho^\epsilon}-1)\sqrt{\rho^\epsilon}\u^\epsilon\cdot
 \u \big](0)\nonumber\\
  &     +\int^t_0\int_{\mathbb{R}^d}(\rho^\epsilon-\sqrt{\rho^\epsilon})
  \u^\epsilon\cdot ((\u\cdot \nabla )\u),\\
S^{\epsilon}_4(t)= &
\int^t_0\!\int_{\mathbb{R}^d}(\mathbf{Z}^\epsilon \cdot \nabla) \H
\cdot [(1-\sqrt{\rho^\epsilon})\u^\epsilon]  - \int^t_0
\!\int_{\mathbb{R}^d}\big\{[(1-\sqrt{\rho^\epsilon})\u^\epsilon
]\cdot \nabla\big\}\H  \cdot   \mathbf{Z}^\epsilon, \nonumber\\
S^{\epsilon}_5(t)= & \int^t_0\int_{\mathbb{R}^d}(1-\rho^\epsilon)
\u^\epsilon\cdot [(\H\cdot \nabla) \H)]
+\frac12 \int^t_0\int_{\mathbb{R}^d}(\rho^\epsilon-1)\u^\epsilon\cdot \nabla (| \H|^2),\\
S^{\epsilon}_6(t)=&
    \int^t_0 \int_{\mathbb{R}^d}\rho^\epsilon \u^\epsilon \cdot \nabla p,\\
S^{\epsilon}_7(t)=& -\int^t_0\int_{\mathbb{R}^d}(\sqrt{
\rho^\epsilon} \u^\epsilon - \u) \cdot \nabla
\big(\frac{|\u|^2}{2}\big).
\end{align*}

\medskip
\emph{Step 4: Convergence of the  modulated energy functional.}

To show the convergence of the modulated energy functional and to
finish our proof,
 we have to  estimate the reminders $ S^{\epsilon,\delta}_1(t)$ and  $
 S^\epsilon_i(t), i=2,\dots,7$. We remark that the terms $ S^{\epsilon}_3(t)$,
 $S^{\epsilon}_6(t)$ and $S^{\epsilon}_7(t)$ where the magnetic
field is not included are estimated in the similar way as in
\cite{M01b}.

First in view of \eqref{L2g} and the following two elementary
inequalities
\begin{align}
&|\sqrt{x}-1|^2\leq M|x-1|^\gamma,\;\;\;|x-1|\geq\delta, \;\;
\gamma\geq 1,\label{ine1}\\
&|\sqrt{x}-1|^2\leq M|x-1|^2,\;\;\;x\geq 0 \label{ine2}
\end{align}
for some positive constant $M$ and $0<\delta<1$, we obtain
\begin{align}
 \int_{\mathbb{R}^d}|\sqrt{\rho^\epsilon}-1|^2
=&\int_{|\rho^\epsilon-1|\leq
1/2}|\sqrt{\rho^\epsilon}-1|^2+\int_{|\rho^\epsilon-1|\geq
1/2}|\sqrt{\rho^\epsilon}-1|^2\nonumber\\
\leq&M\int_{|\rho^\epsilon-1|\leq
1/2}|\rho^\epsilon-1|^2+M\int_{|\rho^\epsilon-1|\geq
1/2}|\rho^\epsilon-1|^\gamma\nonumber\\
\leq &M\epsilon^2.\label{ine3}
\end{align}
Then, by using H\"older inequality, the estimates \eqref{ine3} and
\eqref{cc}, the assumption on the initial data, the strong
convergence of $\rho^\epsilon$, the estimates on $\u^\epsilon$,
$\sqrt{\rho^\epsilon}\u^\epsilon$ and $\H^\epsilon$, and the
regularities of $\u$ and $\H$,   we can show that
$S^{\epsilon,\delta}_1(t)$ converges to $0$ for all fixed $\delta$
and almost all $t$, uniformly in $t$ when $\epsilon $ goes to $ 0$.
  The similar arguments can be found in \cite{DG99} and
\cite{HW3}.

Next, we begin to estimate the terms $S^\epsilon_i(t),i=2,\dots, 7$.
For the term $S^{\epsilon}_2(t)$, by Young's inequality and the
regularity of $\u$ and $\H$, we have
\begin{align}\label{rr1}
|S^{\epsilon}_2(t)|\leq \frac{\mu^\epsilon}{2}
\int^t_0\int_{\mathbb{R}^d}|\nabla \u^\epsilon|^2 +
\frac{\nu^\epsilon}{2}\int^t_0\int_{\mathbb{R}^d}|\nabla
\H^\epsilon|^2 +C_T\mu^\epsilon+C_T\nu^\epsilon.
\end{align}

For the  term  $S^{\epsilon}_3(t)$, by H\"older's inequality, the
estimate \eqref{ine3},
  the assumption on the initial data, the estimate on
$\sqrt{\rho^\epsilon}\u^\epsilon$, and the regularity of $\u$, we
infer that
\begin{align}\label{rr2}
|S^{\epsilon}_3(t)|\leq  & C\epsilon +\|\u(t)\|_{L^\infty}
\Big(\int_{\mathbb{R}^d} |\sqrt{\rho^\epsilon}-1|^2 \Big)^{\frac12}
\Big(\int_{\mathbb{R}^d} \rho^\epsilon |\u^\epsilon|^2\Big)^{\frac12}\nonumber\\
&+ \|[(\u\cdot\nabla)\u](t)\|_{L^\infty}
\Big(\int^t_0\int_{\mathbb{R}^d}
|\sqrt{\rho^\epsilon}-1|^2\Big)^{\frac12}
\Big(\int^t_0\int_{\mathbb{R}^d} \rho^\epsilon |\u^\epsilon|^2 \Big)^{\frac12}\nonumber\\
\leq & C_T\epsilon.
\end{align}

For the term $S^{\epsilon}_4(t)$, making use of  the basic
inequality \eqref{be},  the regularity of  $\H$, and H\"{o}lder's
inequality, we get
\begin{align}\label{rr3}
|S^\epsilon_4(t)|\leq &(\|[(\H\cdot \nabla)\H](t)\|_{L^\infty}
+ \|\nabla \H(t)\|_{L^\infty}\cdot\|\H(t)\|_{L^\infty})\nonumber\\
&
\times\Big(\int^t_0\int_{\mathbb{R}^d}|\sqrt{\rho^\epsilon}-1|^2\Big)^{\frac12}
\Big(\int^t_0\int_{\mathbb{R}^d}|\u^\epsilon|^2\Big)^{\frac12}\nonumber\\
& + \|\nabla
\H(t)\|_{L^\infty}\int^t_0\Big[\Big(\int_{\mathbb{R}^d}|\sqrt{\rho^\epsilon}-1|^2\Big)^{\frac12}
  \|\u^\epsilon(\tau)\|_{L^6}  \|\H^\epsilon(\tau)\|_{L^3}\Big]d\tau\nonumber\\
\leq &
C_T\epsilon\Big(\int^t_0\int_{\mathbb{R}^d}|\u^\epsilon|^2\Big)^{\frac12}
+C_T\epsilon \int^t_0
  \|\u^\epsilon(\tau)\|_{L^6}  \|\H^\epsilon(\tau)\|_{L^3} d\tau \equiv Y^\epsilon(t).
\end{align}
Note that there is no uniform estimate on
$||\u^\epsilon||_{L^2(0,T;L^2)}$ in $Y^\epsilon(t)$ (see the
estimate \eqref{uepsilon}). We have to deal with $Y^\epsilon(t)$ in
dimension two and three, respectively. In the case of $d=2$, using
Sobolev's imbedding, the basic inequality \eqref{be}, the assumption
\eqref{abc}, and the inequality \eqref{uepsilon}, we obtain
\begin{align}\label{rr31}
Y^\epsilon(t) \leq & C_T\epsilon
\big[1+\epsilon(\mu^\epsilon)^{-\frac12}\big]+C_T\epsilon
\Big(\int^t_0\|\u^\epsilon(\tau)\|^2_{H^1(\mathbb{R}^2)}d\tau\Big)^{\frac12}
\Big(\int^t_0\|\H^\epsilon(\tau)\|^2_{H^1(\mathbb{R}^2)}d\tau\Big)^{\frac12}\nonumber\\
=
& C_T\epsilon
\big[1+\epsilon(\mu^\epsilon)^{-\frac12}\big]+C_T\epsilon
(\|\u^\epsilon\|_{L^2(0,T;L^2(\mathbb{R}^2))}
+\|\nabla\u^\epsilon\|_{L^2(0,T;L^2(\mathbb{R}^2))})\nonumber\\
& \quad \times (\|\H^\epsilon\|_{L^2(0,T;L^2(\mathbb{R}^2))}+\|\nabla\H^\epsilon\|_{L^2(0,T;L^2(\mathbb{R}^2))})\nonumber\\
\leq & C_T\epsilon
+C_T\epsilon\big[1+(1+\epsilon)(\mu^\epsilon)^{-\frac12}\big]
\cdot \big[1+(\nu^\epsilon)^{-\frac12}\big]\nonumber\\
\leq & C_T \epsilon +C_T \epsilon^{\sigma}\leq C_T
\epsilon^{\sigma},
\end{align}
where $\sigma =1-(\alpha+\beta)/2>0$ due to $0<\alpha+\beta<2$. For
$d=3$, using Sobolev's imbedding, the basic inequality \eqref{be},
the assumption \eqref{abc}, the  inequality \eqref{uepsilon}, and
$\|\u^\epsilon\|_{L^6(\mathbb{R}^3)}\leq C\|\nabla
\u^\epsilon||_{L^2(\mathbb{R}^3)}$,  we obtain
\begin{align}\label{rr32}
 Y^\epsilon(t) \leq & C_T\epsilon
\big[1+\epsilon^{\frac23}(\mu^\epsilon)^{-\frac12}\big]\nonumber\\
& +C_T\epsilon
\Big(\int^t_0\|\nabla\u^\epsilon(\tau)\|^2_{L^2(\mathbb{R}^3)}d\tau\Big)^{\frac12}
\Big(\int^t_0\|\H^\epsilon(\tau)\|^2_{H^1(\mathbb{R}^3)}d\tau\Big)^{\frac12}\nonumber\\
=
& C_T\epsilon
\big[1+\epsilon^{\frac23}(\mu^\epsilon)^{-\frac12}\big]\nonumber\\
& +C_T\epsilon \|\nabla\u^\epsilon\|_{L^2(0,T;L^2(\mathbb{R}^3))}
 (\|\H^\epsilon\|_{L^2(0,T;L^2(\mathbb{R}^3))}+\|\nabla\H^\epsilon\|_{L^2(0,T;L^2(\mathbb{R}^3))})\nonumber\\
\leq &  C_T\epsilon  +C_T\epsilon^{\frac53}
(\mu^\epsilon)^{-\frac12}+ C_T\epsilon (\mu^\epsilon)^{-\frac12}
  \big(1+(\nu^\epsilon)^{-\frac12}\big)\nonumber\\
\leq & C_T \Big(\epsilon +
\epsilon^{\frac53-\frac{\alpha}{2}}+\epsilon
^{1-\frac{\alpha+\beta}{2}}\Big)\leq C_T \epsilon^{\sigma},
\end{align}
where $\sigma =1-(\alpha+\beta)/2>0$.

For the term $S^{\epsilon}_5(t)$, one can utilize the inequality
\eqref{ine3}, the  inequality \eqref{uepsilon}, the estimate on
 $\sqrt{\rho^\epsilon}\u^\epsilon$, the regularity
of  $\H$, and $\rho^\epsilon-1=\rho^\epsilon-\sqrt{\rho^\epsilon}
+\sqrt{\rho^\epsilon} -1$ to deduce
\begin{align}\label{rr4}
|S^\epsilon_5(t)|\leq &(\|[(\H\cdot \nabla)\H](t)\|_{L^\infty}+
\|\nabla (|\H|^2)\|_{L^\infty})
\Big(\int^t_0\int_{\mathbb{R}^d}|\sqrt{\rho^\epsilon}-1|^2\Big)^{\frac12}\nonumber\\
& \times\bigg[
\Big(\int^t_0\int_{\mathbb{R}^d}|\u^\epsilon|^2\Big)^{\frac12}+
\Big(\int^t_0\int_{\mathbb{R}^d}\rho^\epsilon|\u^\epsilon|^2\Big)^{\frac12} \bigg]\nonumber\\
\leq & \left\{
         \begin{array}{ll}
       C_T\epsilon
\big[1+\epsilon (\mu^\epsilon)^{-\frac12}\big] = C_T\epsilon +C_T\epsilon^{2-\frac{\alpha}{2}}\leq C_T\epsilon , & {d=2;} \\
        C_T\epsilon
\big[1+\epsilon^{\frac23}(\mu^\epsilon)^{-\frac12}\big]=C_T\epsilon+C_T\epsilon^{\frac53-\frac{\alpha}{2}}
, \quad & {d=3.}
         \end{array}
       \right.
\end{align}
Here $2-{\alpha}/{2}>1$ and  $5/3-{\alpha}/{2}>0$ since
$0<\alpha<2$.

Using \eqref{ba}, \eqref{L2g} and \eqref{re}, the term
$S^{\epsilon}_6(t)$ can be bounded as follows.
\begin{align}\label{rr5}
 | S^{\epsilon}_6(t)|=&
  \Big|\int^t_0 \int_{\mathbb{R}^d}\rho^\epsilon \u^\epsilon \cdot  \nabla p\Big|\nonumber\\
 =&  \Big| \int_{\mathbb{R}^d}\big\{((\rho^\epsilon-1) p)(t)-((\rho^\epsilon-1)
 p)(0)\big\}- \int^t_0\int_{\mathbb{R}^d}(\rho^\epsilon-1)  \partial_t p \Big|\nonumber\\
  \leq &   \Big(\int_{|\rho^\epsilon-1|\leq 1/2}|\rho^\epsilon-1|^2 \Big)^{\frac12}
  \Big[\Big(\int_{\mathbb{R}^d}|p(t)|^2\Big)^{\frac12}+\Big(\int_{\mathbb{R}^d}|p(0)|^2\Big)^{\frac12}\Big]\nonumber\\
  & +  \Big(\int_{|\rho^\epsilon-1|\geq 1/2}|\rho^\epsilon-1|^\gamma \Big)^{\frac{1}{\gamma}}
  \Big[\Big(\int_{\mathbb{R}^d}|p(t)|^{\frac{\gamma}{\gamma-1}}\Big)^{\frac{\gamma-1}{\gamma}}
  +\Big(\int_{\mathbb{R}^d}|p(0)|^{\frac{\gamma}{\gamma-1}}\Big)^{\frac{\gamma-1}{\gamma}}\Big]\nonumber\\
&+\int^t_0 \Big(\int_{|\rho^\epsilon-1|\leq
 1/2}|\rho^\epsilon-1|^2\Big)^{\frac12}\Big(\int_{\mathbb{R}^d}|\partial_tp(t)|^2
\Big)^{\frac12}
\nonumber\\
& + \int^t_0 \Big(\int_{|\rho^\epsilon-1|\geq  1/2}
|\rho^\epsilon-1|^{\gamma}\Big)^{{\frac{1}{\gamma}}}
\Big(\int_{\mathbb{R}^d}| \partial_t p(t)|^{\frac{\gamma}{\gamma-1}} \Big)^{{\frac{\gamma-1}{\gamma}}}\nonumber\\
\leq & C_T(\epsilon+\epsilon^{2/\kappa})\leq C_T \epsilon,
\end{align}
where $\kappa=\min\{2,\gamma\}$  and we have used the conditions
$s>d/2+2$ and $\gamma>1$.

Finally, to estimate the term  $ S^\epsilon_7(t)$, we rewrite it as
\begin{align}\label{rr6}
  S^{\epsilon}_7(t)=& -\int^t_0\int_{\mathbb{R}^d}(\sqrt{ \rho^\epsilon}\u^\epsilon-
 \u) \cdot \nabla \big(\frac{|\u|^2}{2}\big)\nonumber\\
  =&\int^t_0\int_{\mathbb{R}^d}\sqrt{\rho^\epsilon}(\sqrt{\rho^\epsilon}-1)
 \u^\epsilon \cdot \nabla \big(\frac{|\u|^2}{2}\big)
  -\int^t_0\int_{\mathbb{R}^d}\rho^\epsilon \u^\epsilon \cdot \nabla
 \big(\frac{|\u|^2}{2}\big)\nonumber\\
  =& S^{\epsilon}_{71}(t)+S^{\epsilon}_{72}(t),
\end{align}
where
\begin{align*}
  S^{\epsilon}_{71}(t)& = \int^t_0\int_{\mathbb{R}^d}\sqrt{\rho^\epsilon}(\sqrt{\rho^\epsilon}-1) \u^\epsilon
 \cdot
\nabla \big(\frac{|\u|^2}{2}\big),\\
S^{\epsilon}_{72}(t) &=
   \int^t_0\int_{\mathbb{R}^d}(\rho^\epsilon-1)
 \partial_t\big(\frac{|\u|^2}{2}\big)
   - \int_{\mathbb{R}^d}\Big[\Big((\rho^\epsilon-1)
 \big(\frac{|\u|^2}{2}\big)\Big)(t) -\Big((\rho^\epsilon-1) \big(\frac{|\u|^2}{2}\big)\Big)(0)\Big].
\end{align*}
 Applying arguments similar to those used for $S^{\epsilon}_{3}(t)$
and $S^{\epsilon}_{6}(t)$, we arrive at the following boundedness
\begin{align}\label{rr66}
 |S^{\epsilon}_{7}(t)|\leq  |S^{\epsilon}_{71}(t)|+|S^{\epsilon}_{72}(t)|\leq C_T \epsilon.
\end{align}
Thus   by collecting all estimates above and applying  the Gronwall's inequality  we deduce that,  for almost all
$t\in (0,T)$,
\begin{align}\label{cq}
 &
 \|\mathbf{w}^{\epsilon,\delta}(t)\|^2_{L^2}+
 \|\mathbf{Z}^{\epsilon}(t)\|^2_{L^2}+\|\Pi^\epsilon(t)
 -\phi^{\epsilon,\delta}(t)\|^2_{L^2}\nonumber\\
\leq & \bar C
\Big[\|\mathbf{w}^{\epsilon,\delta}(0)\|^2_{L^2}+\|\mathbf{Z}^{\epsilon}(0)\|^2_{L^2}
+ \|\Pi^\epsilon_0-\phi^{\epsilon,\delta}(0)\|^2_{L^2}
+C_T\epsilon^\sigma +\sup_{0\leq s\leq
t}S^{\epsilon,\delta}_1(s)\Big],
\end{align}
where $$\bar C=\exp {\Big\{C\int^T_0\big[\|\nabla \u(s)\|_{L^\infty}
+ \|\nabla\H(s)\|_{L^\infty}\big]ds\Big\}}<+\infty.$$

 Noticing that the projection $P$ is a bounded linear mapping
from $L^2$ to $L^2$, we obtain that
\begin{align}
\sup_{0\leq t\leq
T}\|P(\sqrt{\rho^\epsilon}\u^\epsilon)-\u\|_{L^2}&=\sup_{0\leq t\leq
T}\Big\|P\Big(\sqrt{\rho^\epsilon}\u^\epsilon-\u-\g^{\epsilon,\delta}\Big)\Big\|_{L^2}\nonumber\\
&\leq \sup_{0\leq t\leq
T}\|\sqrt{\rho^\epsilon}\u^\epsilon-\u-\g^{\epsilon,\delta}\|_{L^2}.
\end{align}
 Thanks to  \eqref{cq}, we further  obtain that
\begin{align*}
 &\overline{\lim}_{\epsilon \rightarrow
 0}\|P(\sqrt{\rho^\epsilon}\u^\epsilon)-\u\|_{L^\infty(0,T; L^2)}\\
 \leq & C\bar C \lim_{\delta \rightarrow 0}\Big[\|\mathbf{J}_0-\u_0-Q(\mathbf{J}_0)\ast
 \chi_\delta \|^2_{L^2}+\|\varphi_0-\varphi_0\ast \chi_\delta
 \|^2_{L^2}\Big]
=  0.
\end{align*}
The uniform convergence (in $t$) of
$P(\sqrt{\rho^\epsilon}\u^\epsilon)$ to $\u$ in $L^2(\mathbb{R}^d)$
is proved.

Now we claim that $\mathbf{J}=\u$ and $\mathbf{K}=\H$. Since
$\g^{\epsilon,\delta}$ tends weakly-$\ast$ to zero by \eqref{cc},
$\mathbf{w}^{\epsilon,\delta}$ tends  weakly-$\ast$ to  $\mathbf{J}
- \u$. It is easy to see that $\mathbf{Z}^\epsilon$ tends
weakly-$\ast$ to $\mathbf{K} - \H$.  Hence by \eqref{cq},  we obtain
that
\begin{align}\label{cr}
 &    \|\mathbf{J}-\u\|^2_{L^\infty(0,T; L^2)}+
\|\mathbf{K}-\H\|^2_{L^\infty(0,T; L^2)}\nonumber\\
 & \quad \leq  C\bar C \Big[\|\mathbf{J}_0-\u_0-Q(\mathbf{J}_0)\ast \chi_\delta \|^2_{L^2}
 +\|\varphi_0-\varphi_0\ast \chi_\delta \|^2_{L^2}\Big],
\end{align}
where we have used the fact that $\|
\text{weak-}\ast\text{--}\lim\theta^\epsilon\|_{L^\infty(0,T;L^2)}\leq
\varlimsup_{\epsilon\rightarrow
0}\|\theta^\epsilon||_{L^\infty(0,T;L^2)}$ for any
$\theta^\epsilon\in L^\infty(0,T;L^2)$. Thus  we deduce
$\mathbf{J}=\u$ and $\mathbf{K}=\H$   in $L^\infty(0,T; L^2)$ by
letting $\delta$ go to $0$.

Finally, we show the local strong convergence of
$\sqrt{\rho^\epsilon}\u^\epsilon$ to $ \u$ in $L^r(0,T;
L^2(\Omega))$ for all $1\leq r<+\infty$ on any bounded domain
$\Omega \subset \mathbb{R}^d$. In fact, for all $t\in [0,T]$, we
have
\begin{align*}
\|\sqrt{\rho^\epsilon}\u^\epsilon-\u\|_{L^2(\Omega)}& \leq C(\Omega)
 \|\sqrt{\rho^\epsilon}u^\epsilon-\u-\g^{\epsilon,\delta}\|_{L^2(\Omega)}
 +C(\Omega)\|\g^{\epsilon,\delta}\|_{L^2(\Omega)}\\
&\leq
C(\Omega)\|\sqrt{\rho^\epsilon}\u^\epsilon-\u-\g^{\epsilon,\delta}\|_{L^2(\Omega)}
 +C(\Omega)\|\g^{\epsilon,\delta}\|_{L^q(\Omega)}
\end{align*}
for any $q>2$. Using the estimate  \eqref{cd}, we can take the
limit on $\epsilon $ and then on $\delta $ as above to deduce that
 $\sqrt{\rho^\epsilon}\u^\epsilon$ converges to $ \u$ in $L^r(0,T;
 L^2(\Omega))$. Thus we complete our proof.
\end{proof}


{\bf Acknowledgement:} The authors are very grateful to the referees
for their constructive comments and helpful suggestions, which
considerably improved the earlier version of this paper. This work
was partially done when Fucai Li visited the Institute of Applied
Physics and Computational Mathematics. He would like to thank the
institute for hospitality.
 Jiang was supported by the National Basic Research Program (Grant
No. 2005CB321700) and NSFC (Grant No. 40890154).
   Ju was supported by NSFC(Grant No. 10701011).
Li was supported  by NSFC (Grant No. 10971094).


\end{document}